\documentclass{article}
\usepackage{indentfirst,overpic}
\usepackage{amsthm}
\usepackage{amsfonts}
\usepackage{amssymb}
\usepackage{overpic}

\newtheorem{remark}{Remark}

\newcommand{\COLORON}{1}
\newcommand{\NOTESON}{0}
\newcommand{\Debug}{0}

\newcommand{\mue}{\ensuremath{<}-universal element}
\newcommand{\kriz}{K\v{r}\`i\v{z}}
\newcommand{\forb}[1]{\mathrm{Forb}(#1)}
\newcommand{\ex}[1]{\mathrm{Ex}(#1)}

\newcommand{\cof}{co-finite}
\newcommand{\Cof}{Co-finite}
\newcommand{\finy}{co-finite}

\usepackage[usenames]{color} 
\usepackage{amsthm,amssymb,amsmath,bbm,enumerate,graphicx,epsf,stmaryrd,accents}
\usepackage[bookmarks, colorlinks=false, breaklinks=true]{hyperref} %REMOVE FOR ARXIV 

\usepackage{authblk}

\renewcommand{\theenumi}{(\roman{enumi})}

\newcommand{\comment}[1]{}
\newcommand{\COMMENT}[1]{}

\definecolor{darkgray}{rgb}{0.3,0.3,0.3}
\newcommand{\defi}[1]{{\color{darkgray}\emph{#1}}}

\comment{
	\begin{lemma}\label{}	
\end{lemma}
\begin{proof}

\end{proof}

\begin{theorem}\label{}
\end{theorem} 
\begin{proof} 	

\end{proof}

}

\newtheorem{proposition}{Proposition}[section]
\newtheorem{definition}[proposition]{Definition}
\newtheorem{theorem}[proposition]{Theorem}
\newtheorem{corollary}[proposition]{Corollary}

\newtheorem{lemma}[proposition]{Lemma}
\newtheorem{observation}[proposition]{Observation}
\newtheorem{conjecture}{{Conjecture}}[section]

\newtheorem{problem}[conjecture]{{Problem}}
\newtheorem*{noProblem}{{Problem}}
\newtheorem{question}[conjecture]{{Question}}

\newtheorem{examp}[proposition]{Example}%[section]

\newcommand{\kreis}[1]{\mathaccent"7017\relax #1}

\newcommand{\FIG}{0}

\ifnum \NOTESON = 1 \newcommand{\note}[1]{ 

\hspace*{-30pt}
	{\color{blue}  NOTE: \color{Turquoise}{\small  \tt \begin{minipage}[c]{1.1\textwidth}  #1 \end{minipage} \ignorespacesafterend }} 
	
	}
\else \newcommand{\note}[1]{} \fi

\newcommand{\afsubm}[1]{ \ifnum \Debug = 1 {\mymargin{#1}}
\fi} %For notes on after-submission changes

\ifnum \Debug = 1 
\else  \fi

\ifnum \FIG = 1 \newcommand{\fig}[1]{Figure ``{#1}''}
\else \newcommand{\fig}[1]{Figure~\ref{#1}} \fi

\ifnum \FIG = 1 
\else  \fi

\ifnum \Debug = 1 \usepackage[notref,notcite]{showkeys}
\fi

\ifnum \COLORON = 0 \renewcommand{\color}[1]{}
\fi

\newcommand{\N}{\ensuremath{\mathbb N}}
\newcommand{\R}{\ensuremath{\mathbb R}}

\newcommand{\Q}{\ensuremath{\mathbb Q}}
\newcommand{\BS}{\ensuremath{\mathbb S}}

\newcommand{\cc}{\ensuremath{\mathcal C}}

\newcommand{\ce}{\ensuremath{\mathcal E}}
\newcommand{\cf}{\ensuremath{\mathcal F}}

\newcommand{\ct}{\ensuremath{\mathcal T}}
\newcommand{\cu}{\ensuremath{\mathcal U}}

\newcommand{\oo}{\ensuremath{\omega}}

\newcommand{\sig}{\ensuremath{\sigma}}

\newcommand{\sm}{\backslash}
\newcommand{\restr}{\upharpoonright}

\newcommand{\isom}{\cong}

\newcommand{\cls}[1]{\ensuremath{\overline{#1}}}

\makeatletter
\DeclareRobustCommand{\cev}[1]{%
  \mathpalette\do@cev{#1}%
}
\newcommand{\do@cev}[2]{%
  \fix@cev{#1}{+}%
  \reflectbox{$\m@th#1\vec{\reflectbox{$\fix@cev{#1}{-}\m@th#1#2\fix@cev{#1}{+}$}}$}%
  \fix@cev{#1}{-}%
}
\newcommand{\fix@cev}[2]{%
  \ifx#1\displaystyle
    \mkern#23mu
  \else
    \ifx#1\textstyle
      \mkern#23mu
    \else
      \ifx#1\scriptstyle
        \mkern#22mu
      \else
        \mkern#22mu
      \fi
    \fi
  \fi
}

\makeatother
\newcommand{\nin}{\ensuremath{{n\in\N}}}

\newcommand{\pth}[2]{\ensuremath{#1}\text{--}\ensuremath{#2}~path}

\newcommand{\seq}[1]{\ensuremath{(#1_n)_{n\in\N}}} 

\newcommand{\g}{\ensuremath{G\ }}
\newcommand{\G}{\ensuremath{G}}

\newcommand{\Lr}[1]{Lemma~\ref{#1}}

\newcommand{\Tr}[1]{Theorem~\ref{#1}}

\newcommand{\Sr}[1]{Section~\ref{#1}}

\newcommand{\Prr}[1]{Pro\-position~\ref{#1}}

\newcommand{\Cr}[1]{Corollary~\ref{#1}}
\newcommand{\Cnr}[1]{Con\-jecture~\ref{#1}}

\newcommand{\Dr}[1]{De\-fi\-nition~\ref{#1}}
\newcommand{\Qr}[1]{Question~\ref{#1}}

\newcommand{\lf}{locally finite}

\renewcommand{\iff}{if and only if}
\newcommand{\fe}{for every}

\newcommand{\st}{such that}

\newcommand{\ti}{there is}
\newcommand{\ta}{there are}

\newcommand{\obda}{without loss of generality}

\newcommand{\wrt}{with respect to}

\newcommand{\labtequ}[2]{%\labtequc{#1}{#2}}
 \begin{equation} \label{#1} 	\begin{minipage}[c]{0.9\textwidth}  #2 \end{minipage} \ignorespacesafterend \end{equation} }

\newcommand{\mymargin}[1]{% <- dieses % verhindert ein ungewolltes Leerzeichen
 \ifnum \Debug = 1
  \marginpar{%
    \begin{minipage}{\marginparwidth}\small%
      \begin{flushleft}%
        {\color{blue}#1}%
      \end{flushleft}%
   \end{minipage}%
  }%
 \fi
}%

\newcommand{\mySection}[2]{}

\newcommand{\RSTc}{Robertson--Seymour theorem \cite{GMXX}}

\newcommand{\DK}{Diestel and K\"uhn}

\newcommand{\DB}{\cite{diestelBook05}}
\begin{document}
	\title{On graph classes with minor-universal elements}
	
\author{Agelos Georgakopoulos\thanks{Supported by  EPSRC grants EP/V048821/1 and EP/V009044/1.}}
\affil{  {Mathematics Institute, University of Warwick}\\
  {\small CV4 7AL, UK}}

\date{\today}

\maketitle

\begin{abstract}
A graph $U$ is universal for a graph class $\mathcal{C}\ni U$, if every $G\in \mathcal{C}$ is a minor of $U$. We prove the existence or absence of universal graphs in several natural graph classes, including graphs component-wise embeddable into a surface, and graphs forbidding $K_5$, or $K_{3,3}$, or $K_\infty$ as a minor. We prove the existence of uncountably many minor-closed classes of countable graphs that (do and) do not have a universal element. 

Some of our results and questions may be of interest to the finite graph theorist. In particular, one of our side-results is that every $K_5$-minor-free graph is a minor of a $K_5$-minor-free graph of maximum degree 22.
\end{abstract}

\medskip



\section{Introduction}

Given a binary relation $\prec$ between graphs ---typically the (induced) subgraph or minor relation--- we say that a graph $U$ is $\prec$-universal for a graph class $\mathcal{C}\ni U$, if $G \prec U$ holds for every $G\in \mathcal{C}$. (It is often, but not always, the case that $\cc$ is closed under $\prec$.)

The study of universal graphs is one of the oldest topics of infinite graph theory. The best known instance is the 
Rado graph \cite{RadUni}, an induced-subgraph-universal countable graph. %De Bruijn \cite[THEOREM 2]{RadUni} proved that the class of countable locally finite graphs has no subgraph-universal element. 
 %Henson \cite{Henson} constructed an induced-subgraph-universal countable graph with no $K_n$ subgraph \fe\ $n\geq 3$. 
Diestel, Halin \& Vogler \cite{DiHaVoSom} proved that there is no subgraph-universal graph in the class of countable graphs with no $K_n$ minor for $5 \leq n \leq \aleph_0$, but there is one for $n\leq 4$. %The question of whether the same classes have minor-universal graphs remained open. It was stated by Diestel \& K\"uhn \cite{DieKuhUni} in the more general formulation for any forbidden minor, with the case $K^\infty$ emphasized. 

Pach \cite{PacPro} proved that there is no subgraph-universal planar graph. Huynh et al.\ \cite{HMSTW} recently strengthened this by showing that if $G$ contains every countable planar graph as a subgraph then it must have the infinite complete graph $K_\infty$ as a minor. Lehner \cite{LehUni} strengthened this further by showing that the same conclusion holds if \G\  contains every countable planar graph as a topological minor.

For more background about universal graphs we refer to e.g.\ \cite{Henson, KoMePaSom, KriUni, KuhMin} and the survey \cite{KomPacUni}.

\medskip
This paper studies $<$-universal graphs, where $<$ stands for the minor relation, with an emphasis on interactions with the finite graph minor theory. Despite Pach's result and its aforementioned strengthenings, \DK\ proved  
\begin{theorem}[\cite{DieKuhUni}] \label{thm DK}
There is a $<$-universal countable planar graph.
\end{theorem}
\Tr{thm DK} provides motivation as well an important tool for this paper. Other natural minor-closed classes known so far to have minor-universal elements are forests, series-parallel graphs \cite{DiHaVoSom}, outer-planar graphs \cite{DieKuhUni}, and graphs with accumulation-free embeddings in $\R^2$ \cite{KuhMin}. Diestel \& K\"uhn \cite{DieKuhUni} asked 
\begin{question} \label{Q DK}
Which minor-closed graph classes have a countable \mue?
\end{question}

They specifically discussed the special cases of Question~\ref{Q DK} for the classes $Forb(K_5)$ and $Forb(K_\infty)$.  This paper settles these two cases of Question~\ref{Q DK}, as well as many more. Our first result is easy to explain, and it adapts a well-known idea (\cite[Proposition 3.2.]{KomPacUni}):

\begin{theorem}\label{Kinfty}
The class of countable $K_\infty$-minor-free graphs has no \mue.
\end{theorem}
% *** ---- *** 
\begin{proof}
Suppose $U$ is a universal $K_\infty$-minor-free graph. Add a new vertex $v^*$, and join it with an edge to each vertex of $U$ to obtain a new graph $C(U)$. Notice that $C(U)$ is still  $K_\infty$-minor-free, and therefore we can find $C(U)$ as a minor $M$ of $U$. Let $B_0$ be the branch set of such a minor corresponding to $v^*$. Thus $U$ is a minor of $U \sm B_0$, hence the latter is still universal. Repeat the argument to find a minor of $U \sm B_0$ isomorphic to $C(U)$ every branch set of which contains one of $M$, and let $B_1$ be the new branch set corresponding to $v^*$. Repeat ad infinitum to define \seq{B}, and notice that each $B_i$ sends an edge to all $B_j, j>i$. Thus the $B_i$ form the branch sets of a $K_\infty$ minor in $U$, which is a contradiction.
\end{proof}

We remark that the class of $K_\infty$-minor-free graphs had attracted a lot of attention, and alternative characterizations are known \cite{DieThoExc,RoSeThExcI,RoSeThExcII}.

A second look at the above proof reveals that the choice of the forbidden minor $K_\infty$ is not too important; the same argument can be applied if we replace $K_\infty$ by any graph \g that is a minor of $G - v$ \fe\ $v\in V(G)$. This yields %any of its deck cards 

\begin{corollary}\label{cor}
Let \g be a graph such that $G< G - v$ holds \fe\ $v\in V(G)$. Then the class of $G$-minor-free graphs has no \mue. \qed
\end{corollary}

Examples of such \g include the ray, the half- and full-grid, and any graph which is a disjoint union of infinitely many copies of a fixed graph $H$.  Letting \g be the ray, we thus strengthen a result of 
Diestel, Halin \& Vogler \cite[Theorem~4.1]{DiHaVoSom} saying that there is no subgraph-universal graph in the class of countable rayless  graphs. %We strengthen this in two ways: If \g is a countable graph containing each countable rayless graph as a minor, then $G > K^\infty$.

\medskip
C.~Thomassen observed that for every closed surface $\Sigma$ other than $\BS^2$, the class $\ce_\Sigma$ of countable graphs that embed in $\Sigma$ has no \mue\ \cite[\S~6]{DieKuhUni}. Our next result shows that the reason for this lies in the lack of addability of $\ce_\Sigma$:

\begin{theorem} \label{thm CS Intro}
Let $\Sigma$ be a closed orientable surface, and let $\cc_\Sigma$ be the class of countable graphs $G$ \st\ every component of \g embeds in $\Sigma$. Then $\cc_\Sigma$ has a \mue. 
\end{theorem}

We prove \Tr{thm CS Intro} in \Sr{sec CS} by combining the universal planar graph of \DK\ \cite{DieKuhUni} with a classical result of Youngs \cite{Youngs} about embeddings of finite graphs. We will thereby introduce some handy tools for dealing with infinite embedded graphs, simplifying one of the ingredients of the proof of \Tr{thm DK} in \cite{DieKuhUni}.
\medskip

The other aforementioned question of \DK\ has a positive answer:
\begin{theorem} \label{thm K5 33}
The class of countable graphs with no $K_5$ minor has a \mue. So does the class of countable $K_{3,3}$-minor-free graphs. 
\end{theorem}

This is the most difficult result of this paper. We prove it in \Sr{sec K5}, using again the universal planar graph, combining it with Wagner's clique-sum decompositions of such graphs \cite{Wagner}, as extended to the infinite case by \kriz\ \& Thomas \cite{KriThoCli}.

\medskip
\Tr{Kinfty} provides the simplest proof I am aware of for the fact that \ta\ uncountably many minor-twin classes of countable graphs: for if \seq{G}\ was a countable sequence representing them all, then $\dot{\bigcup}_n G_n \sm \{K_\infty\}$ would be a universal $K_\infty$-minor-free graph. 

Note that \fe\ graph $G$, its minors $\cc_G:=\{H \mid H<G\}$ form a minor-closed family which has $G$ as a \mue. In \Sr{sec Lilian} we will prove that \ta\ uncountably many  minor-closed classes of countable graphs none of which classes has a \mue. This is achieved by combining \Cr{cor} with a beautiful result of Matthiesen \cite{Matthiesen}, which uses Nash-Williams' theorem that trees are well-quasi-ordered to deduce that there are uncountably many topological-minor-twin classes of rooted trees. 

To summarize, we have provided the following response to \Qr{Q DK} which perhaps answers it to the extent possible: 

\begin{theorem} \label{thm unct intro}
There are uncountably many minor-closed classes of countable graphs that do, and uncountably many that do not have a \mue.

\end{theorem}

One can still propose refined versions of the question, e.g.\ asking for sufficient conditions for the existence of a \mue. The following question of this form is motivated by the above discussion on graphs on surfaces and \Cr{cor}:
\begin{noProblem} \label{fin add}
Let $F_1,\ldots, F_k$ be finite connected graphs. Must  $\forb{F_1,\ldots, F_k}$ have a \mue?
\end{noProblem} 
We will discuss this problem further in \Sr{sec Qs}, along with several other open questions.

\bigskip
As a consequence of our proof of \Tr{thm K5 33}, we will obtain (in \Sr{sec Delta}) the following corollary, which as far as I am aware is new even for finite graphs.

\begin{corollary} \label{cor K5}
Every $K_5$-minor-free graph is a minor of a $K_5$-minor-free graph with maximum degree 22. Every $K_{3,3}$-minor-free graph is a minor of a $K_{3,3}$-minor-free graph with maximum degree 9.
\end{corollary}

See \Sr{sec Delta} for a proof and related questions. A proof of \Cr{cor K5} for finite graphs that does not pass through the infinite can be extracted from our methods, repeating about half of the work.

%SSSSSSSSSSSSSSSS
\section{Preliminaries}  \label{prels}

%SSSSSSSSSSSSSSSS
\subsection{Graphs} \label{sec gra}

We follow the terminology of Diestel \DB. We use $V(G)$ to denote the set of vertices, and $E(G)$ the set of edges of a graph \G. For $S\subseteq V(G)$, the subgraph $G[S]$ of \g \defi{induced} by $S$ has vertex set $S$ and contains all edges of \g with both endvertices in $S$.

The \defi{degree} $d(v)=d_G(v)$ of a vertex $v$ in a graph \G, is the number of edges of \g incident with $v$. The \defi{maximum degree} $\Delta(G)$ of \g is $\sup_{v\in V(G)} d(v)$. The \defi{closed neighbourhood} of $v\in V(G)$ is the subgraph of \g induced by $v$ and all its \defi{neighbours}, i.e.\ the vertices sharing an edge with $v$.

We write $G \dot{\cup}  H$  for the disjoint union of two graphs $G,H$, and $\omega \cdot G$ for the disjoint union of countably infinitely many copies of \G.
%sub-cubic

A \defi{ray} is a one-way infinite path.

%SSSSSSSSSSSSSSSS
\subsection{Embeddings} \label{sec emb}

A \defi{surface} is a 2-manifold without boundary. An \defi{embedding}  of a countable graph \g\ into a surface $\Sigma$ is a map $f: G \to \Sigma$ from the 1-complex obtained from \g when identifying each edge with the interval $[0,1]$ to  $\Sigma$  \st\ the restriction of $f$ to each finite subgraph of $G$ is an embedding in the topological sense, i.e.\ a homeomorphism onto its image. (The reason why we restrict to finite subgraphs here is that the 1-complex topology of \g is not metrizable when \g is not \lf, and so such \g cannot have an embedding into a metrizable space $\Sigma$. For example, a star with infinitely many leaves admits an embedding into $\BS^2$ in our sense but it does not admit a topological embedding.)

A \defi{face} of an embedding $f: G \to \Sigma$ is a component of $\Sigma \sm f(G)$. A \defi{facial walk} is a (finite, or two-way infinite) sequence $\ldots v_0 e_0 v_1 e_1\ldots $ of vertices and edges of \g whose $f$-images appear along the boundary of a face of $f$ in that order.

The \defi{(orientable) genus} $\gamma(G)$ of a graph \g is the smallest genus $\gamma(\Sigma)$ of an orientable surface $\Sigma$ into which \g embeds. %This is always well-defined when \g is countable, but may be infinite.

%assignment of a point $f(v)$ to each $v\in V(G)$ and an \arc{f(u)}{f(v)}\ $f(e)$ to each $uv\in E(G)$ so that the interior of $f(e)$ contains no vertex and no point of any other edge.

%SSSSSSSSSSSSSSSS
\subsection{Minors} \label{sec min}

Let $G,H$ be graphs. An $H$-\defi{minor} of $G$ is a collection of disjoint connected subgraphs $B_v, v\in V(H)$ of $G$, called \defi{branch sets}, 
and edges $E_{uv}, uv\in E(H)$ of \g % (called \defi{branch edges}) 
such that each $E_{uv}$ has one end-vertex in $B_u$ and one in $B_v$. We write $H<G$ to express that $G$ has an $H$ minor.

We say that $G$ and $H$ are \defi{minor-twins}, if both $G<H$ and $H<G$ hold. Any two finite minor-twins are isomorphic, but in the infinite case it is often helpful to think in terms of equivalence classes of minor-twins instead of isolated graphs.

Given a set $X$ of graphs, we write $\forb{X}$ for the class of graphs $H$ \st\ no element of $X$ is a minor of $H$. For a minor-closed graph class \cc, we write $\ex{\cc}$ for the set of graphs $G$ \st\ $G\not\in \cc$ but $H\in \cc$ for every $H<G$ \st\ $G\not <H$. If \cc\ consists of finite graphs, then $\forb{\ex{\cc}}=\cc$ holds by the definitions. For infinite graphs this will fail if there is an infinite descending sequence $G_1 > G_2 > \ldots$ (consisting of graphs not in \cc) which as far as I know is open \cite{ThoWel}.

%SSSSSSSSSSSSSSSS
\section{Uncountably many classes with no \mue} \label{sec Lilian}

The aim of this section is to prove

\begin{proposition} \label{prr unct}
There are uncountably many classes of countable graphs that have no \mue.
\end{proposition}

To begin with, recall that \Tr{Kinfty} implies that there are uncountably many minor-twin classes of countable graphs. %; for if there was  a countable sequence \seq{G} with a representative in each minor-twin class, except the class of $K_\infty$, then $\bigcup_n G_n$ would be \mue\ of the class of $K_\infty$-minor free graphs. 

As we will see, one can find uncountably many minor-twin classes among much more restricted classes of graphs using the following result of Matthiesen.

The \defi{tree-order} $\leq$ on the set of vertices of a rooted tree $(T,r)$ is defined by setting $x \leq y$ whenever $x$ lies on the unique path in $T$ from its root $r$ to $y$. Given two rooted trees $(T,r), (T',r')$, we write $(T,r) <_\leq (T',r')$ if there is an embedding $t: (T,r) \to  (T',r')$ that respects the tree-order when restricted to $V(T)$.

\newpage
\begin{theorem}[{\cite{Matthiesen}}] \label{Lilian}
There are uncountably many equivalence classes of locally finite rooted trees  \wrt\ the order-preserving topological minor relation $<_\leq$. This remains true in the class of rooted trees every vertex of which has at most two successors.
\end{theorem}

\Tr{Lilian} provides another way for proving that there are uncountably many minor-twin classes of countable graphs: let \ct\ be an uncountable family of rooted trees no two of which are in the same equivalence class of $<_\leq$ as provided by \Tr{Lilian}, and attach a triangle to the root of each member of \ct. The resulting family of (unrooted) graphs contains no pair of minor-twins. We leave the details to the reader as a warm-up exercise before the following proof.

% *** ---- *** 
\begin{proof}[Proof of \Prr{prr unct}]
Let \ct\ be an uncountable family of rooted trees as provided by \Tr{Lilian}, i.e.\ \st\ no two elements are equivalent \wrt\ $<_\leq$, and every vertex of each $T\in \ct$ has at most two successors in $\leq$.

Our plan is to associate to each $T\in \ct$ a graph class $\cc_T$, defined by forbidding a relative of $T$ as a minor, \st\  $\cc_T$ has no \mue\ by \Cr{cor}. But as $T$ does not have to satisfy the premise of \Cr{cor}, we introduce a graph $T'$ that does using the construction of \fig{figTp}. To make this precise, we start with a ray $R=R_T=r_1 r_2 \ldots$, and attach a new edge $e_i=r_i v_i$ to the $i$th vertex of $R$ \fe\ $i\in \N$. Then, for every even $i$, we attach a triangle to $v_i$, and for every odd $i$, we attach a copy of $T$ identifying the root of $T$ with $v_i$. Let $\ct':= \{T' \mid T \in \ct\}$.
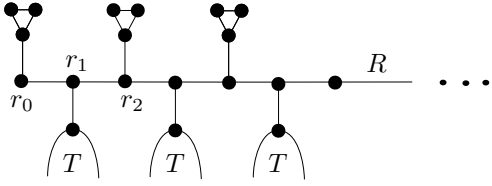
\begin{figure} 
\begin{center}
\begin{overpic}[width=.55\linewidth]{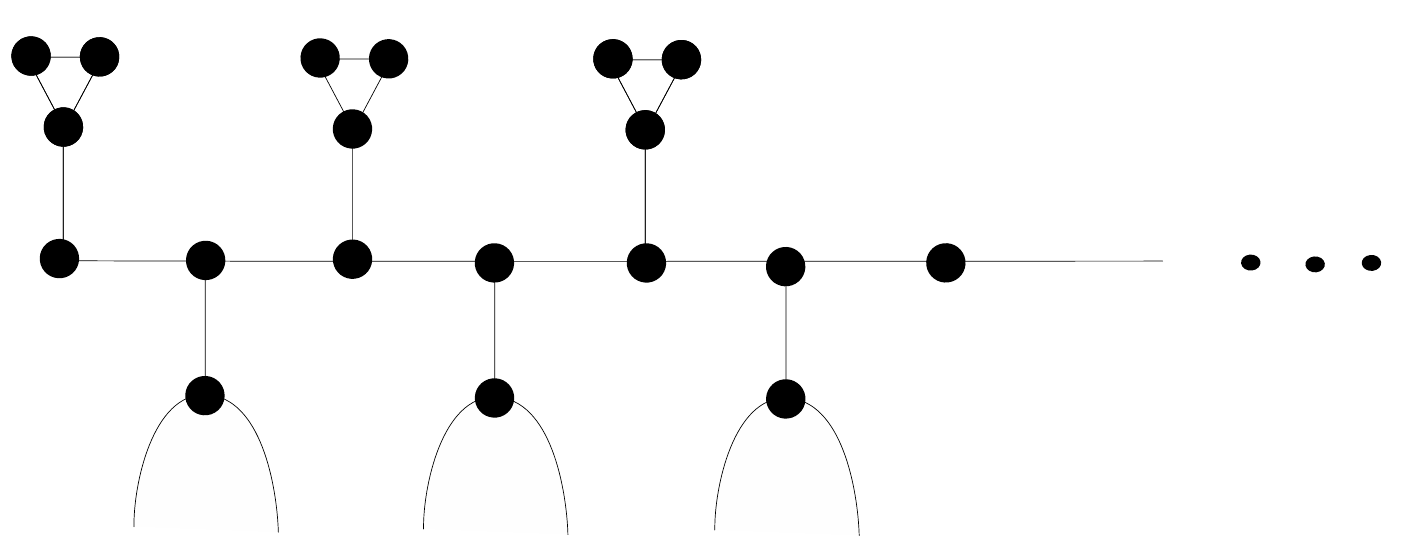} 
\put(2,16){$r_0$}
\put(13,24){$r_1$}
\put(24,16){$r_2$}
\put(73,23){$R$}
\put(12.5,3){$T$}
\put(33,3){$T$}
\put(53.5,3){$T$}
\end{overpic}
\end{center}
\caption{The construction of $T'$.} \label{figTp}
\end{figure}
%\showFig{fig Tp}{The construction of $T'$.}
We claim that
\labtequ{twins}{$\ct'$ contains no pair of minor-twins.}
We emphasize that $\ct'$ is a family of graphs with no designated root, and \eqref{twins} refers to the standard notion of minor, with no reference to $<_\leq$.

Once \eqref{twins} is proven, the above idea can be easily carried out: we define the graph classes $\cc(T):= Forb(T'), T\in \ct$. They are all distinct by  \eqref{twins}, and uncountably many. Note that $T'$ embeds into its subgraph to the right of an arbitrary $r_i$ of $R_T$, and therefore $T'$ satisfies the premise of \Cr{cor}, from which we deduce that $\cc(T)$ has no \mue.

\medskip
To prove \eqref{twins}, suppose that $T'_1 < T'_2$ for some $T'_1, T'_2 \in \ct'$. We will deduce that $T_1 <_\leq T_2$. This immediately implies \eqref{twins}, as $T_1 , T_2$ cannot be twins \wrt\ $<_\leq$. 

To see the latter claim, note first that as all our graphs have maximum degree 3, we can find a {topological} $T'_1$ minor in $T'_2$, i.e. an embedding $g: T'_1 \to T'_2$. Next, we note that $g(R_{T_1})$ must contain all but finitely many edges of $R_{T_2}$, because otherwise $g(R_{T_1})$ could only contain a finite part of $R_{T_2}$, in which case it would be impossible for $g$ to accommodate the infinitely many triangles of $T'_1$. This means that some copy $C_1$ of $T_1$ in $T'_1$ must be mapped by $g$ to the complement of $R_{T_2}$ in $T'_2$, and therefore to a copy $C_2$ of  $T_2$ in $T'_2$. Moreover, the edge $r_i v_i$ incident with $C_1$ must be mapped so that $g(r_i) \leq g(v_i)$ \wrt\ the tree order $\leq$ of $C_2$ (extended to include the vertex of $R_{T_2}$ to which $C_2$ attaches, which vertex serves as the root). This easily implies, by induction of the layers of $C_1$, that $g$ embeds $C_1$ into $C_2$ preserving $\leq$. This proves our claim that $T_1 <_\leq T_2$, establishing \eqref{twins}.

\end{proof}

%SSSSSSSSSSSSSSSS
\section{Universal elements for embeddable graphs} \label{sec CS}

As mentioned in the introduction, C.~Thomassen observed that for every closed surface $\Sigma$ other than $\BS^2$, the class %$\ce_\Sigma$ 
of countable graphs that embed in $\Sigma$ has no \mue. We repeat the argument verbatim from \cite[\S~6]{DieKuhUni}: \\

\begin{minipage}[c]{.91\textwidth} 
{\small  Indeed any such universal graph $U$ would contain a cycle $C$ whose deletion reduces the Euler genus of $U$. Then every minor of $U$ can be embedded in a simpler surface than $\Sigma$ after the deletion of at most $|C|$ vertices. This however will not be the case for every graph embeddable in $\Sigma$.
}
\end{minipage}

\bigskip
Nevertheless, we will prove 

\begin{theorem} \label{thm CS}
Let $\Sigma$ be a closed orientable surface, and let $\cc_\Sigma$ be the class of countable graphs $G$ \st\ every component of \g embeds in $\Sigma$. Then $\cc_\Sigma$ has a \mue.
\end{theorem}

The following will allow us to restrict our attention to \lf\ graphs in the proof of \Tr{thm CS}. Moreover, it allows us to choose our universal graphs to be \lf\ if we want to. A graph is \defi{sub-cubic} if it has maximum degree 3.

\begin{lemma}\label{lem lf}
Let $G$ be a countable graph embeddable in a closed surface $\Sigma$. There is a sub-cubic graph $G'$ embeddable in  $\Sigma$ \st\ $G < G'$.
\end{lemma}

A proof of the special case of \Lr{lem lf} where $\Sigma$ is the plane occupies 4 pages in \cite{DieKuhUni}; our proof is simpler.\footnote{I suspect that it may be possible to obtain a simpler construction of a universal planar graph by modifying the construction of \cite{GeoInf} so as to increase the number of subdivisions in each step.}

The idea is to obtain $G'$ by blowing up each vertex $v$ of \g into a \lf\ tree $T_v$. For this to be possible we need enough space around $v$ in order to embed $T_v$. This motivates the following definition. An embedding $f: G \to \Sigma$ is called \defi{generous around vertices}, if \fe\ $v\in V(G)$ there is a topological open disc $D_v \subset \Sigma$ \st\ $D_v \cap f(V(G)) =f(v)$ and $D_v$ avoids $f(uw)$ for every edge $uw$ with $u,w\neq v$. Similarly, we say that $f$ is \defi{generous around edges}, if \fe\ $e=uv\in E(G)$ there is a topological open disc $D_e \subset \Sigma$ \st\ $D_e \cap f(G) = f(e) \sm \{f(u),f(v)\}$. We call $f$ \defi{generous} if it is generous around both vertices and edges.

\begin{proposition}\label{generous}	
If a countable graph \g admits an embedding into a surface $\Sigma$, then it admits a generous one.
\end{proposition}
% *** ---- *** 
\begin{proof}
Let $f: G \to \Sigma$ be an arbitrary embedding, and let  \seq{H}  be a sequence of finite connected subgraphs  \defi{exhausting} $G$, i.e.\ $\bigcup_{\nin} H_n = G$. We define an embedding $g: G \to \Sigma$ which is generous around vertices by inductively making local alterations to the way $f$ embeds $H_n$. We start by letting $g(H_0) = f(H_0)$. Since $H_0$ is finite, we can choose  disjoint closed discs $D'_v$ around each $f(v), v \in v(H_0)$, each intersecting $f(H_0)$ at a small star around $v$. Let $D_v$ be a disc strictly contained in $D'_v$.  For $n=1,2, \ldots$, having chosen $g(H_{n-1})$ and such discs around each vertex of $H_{n-1}$, we proceed to embed $H_n$. If $f$ embeds any vertices $w$ of $H_n - H_{n-1}$ inside some $D_v$, then we modify the part of $f$ intersecting $D'_v$ so as to push the image of $w$ into the annulus $D'_v \sm D_v$. Some easy topological details are hereby left to the reader.  We treat unwanted images of edges intersecting $D_v$ similarly. Finally, we pick discs $D_w$ around the new vertices $w$, disjoint from each other and any previously chosen $D_v$, and proceed to the next step. The limit embedding $g: \bigcup_n H_n \to \Sigma$ is generous around vertices by definition.

Starting with $g$ instead of $f$, and with a similar approach that this time fixes, for each edge $uv\in E(G)$, an open disc $D'_{uv}$ containing $g(uv) \sm (D_u \cup D_v)$, we transform $g$ into an embedding that is also generous around edges.
\end{proof}

% *** ---- *** 
\begin{proof}[Proof of \Lr{lem lf}]
Let $g: G \to \Sigma$ be a generous embedding (around vertices), as provided by  \Prr{generous}, with corresponding discs $D_v, v\in V(G)$.  Easily, we can assume that these discs are pairwise disjoint, by choosing sub-discs if needed. We will obtain $G'$ by blowing up each vertex $v\in V(G)$ into a tree $T_v$, embedded inside $D_v$. 

To make this precise, let $\mathbb{T}_2$ be the infinite binary tree with root $r$, and attach a leaf $\ell_w$ to each vertex $w$ of $\mathbb{T}_2$ to obtain the tree $T$. Embed $T$ in the closed unit disc $\mathbb{D}$ so that $\ell_w$ is embedded between the two edges of $w$ leading away from $r$ (\fig{figT}). Easily, we may map all the leaves $\ell_w$ of $T$ on the boundary $\partial \mathbb{D}$. Denote this embedding by $f_T: T \to \mathbb{D}$.
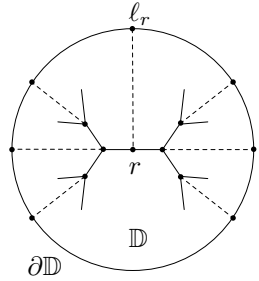
\begin{figure} 
\begin{center}
\begin{overpic}[width=.3\linewidth]{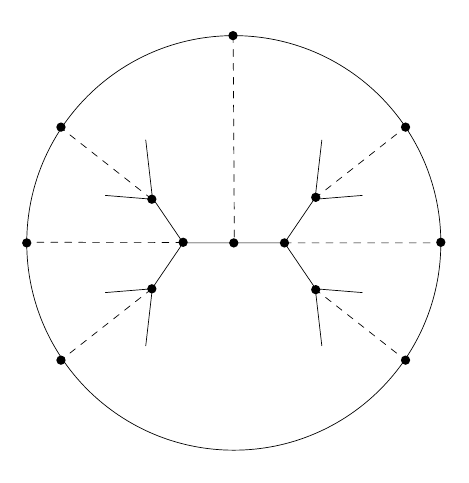} 
\put(47,41){$r$}
\put(47,95.5){$\ell_r$}
\put(11,5){$\partial \mathbb{D}$}
\put(47,15){$ \mathbb{D}$}
\end{overpic}
\end{center}
\caption{The first few layers of the embedding of $T$ in the proof of \Lr{lem lf}. Dashed lines denote  edges leading to leaves.} \label{figT}
\end{figure}

We remark that the ordering of the leaves of $T$ as they appear along the circle $\partial \mathbb{D}$ is of the same type as 
the ordering  of the rationals $\Q \cap [0,1]$ by $<$: between any two elements there is a third one. %(except that the former is a cyclic ordering, but this will not matter below). 

We construct the desired embedded \lf\ graph $G'$ as follows. For every $v\in V(G)$, let $\kreis{D}_v$ be a closed topological disc strictly contained in $D_v$. Delete $g(G) \cap \kreis{D}_v$, and embed a copy $T_v$ of $T$ into $\kreis{D}_v$ by composing $f_T$ with an isomorphism from $\mathbb{D}$ to $\kreis{D}_v$. Next, delete $g(G) \cap D_v \sm \kreis{D}_v$, and note that every edge $uv\in E(v)$ is now mapped by what remains of $g$ to an arc $A^1_u$ from $g(u)$ to a point $u'\in \partial D_v$ plus possibly some arcs with both edpoints on $\partial D_v$, which arcs we also discard. Let $u_1,u_2, \ldots$ be a (potentially finite) enumeration of $N(v)$. For each $i=1,2,\ldots$, choose an arc $A^2_{u_i}$ in $D_v \sm \kreis{D}_v$ joining $u'_i$ to the image of a leaf $\ell_i$ of $T_v$. Note that the union of $A^1_{u_i}$ with $A^2_{u_i}$ and the image $A^3_{u_i}$ of the $\ell_i$--$v$~edge is an arc $A_{u_i v}$ from $u_i$ to $v$, and we think of this arc as a replacement for the $u_i v$~edge of \G. Thus we want to make the above choices so that the $A_{u_i v}$ do not intersect except possibly at their endpoints. In particular, we need to choose pairwise distinct leaves $\ell_i$. These choices are however easy to make greedily: having chosen $A^2_{u_j}, j< i$, we use the remark about the ordering of the leaves of $T_v$ from the previous paragraph, to choose a yet unused leaf $\ell_i$, and an   $u'_i$--$\ell_i$ arc $A^2_{u_i}$ in $D_v \sm \kreis{D}_v$ that avoids  $A^2_{u_j}, j< i$.

Thus we have constructed a \lf\ graph $G'$ embedded into $\Sigma$. By construction, if we contract each $T_v$ into a vertex, then we obtain \G, which proves that $G < G'$ as claimed. Note that $G'$ has maximum degree 4 by construction, and it is easy to modify $T$ and $G'$ to reduce it to 3.
\end{proof}

For the proof of \Tr{thm CS} we will also use the following classical result of Youngs about cellular embeddings.

\begin{theorem}[\cite{Youngs}] \label{Youngs}
Let $\Sigma$ be a closed orientable surface, and let $G$ be a finite graph that embeds into $\Sigma$ but does not embed into a closed orientable surface of smaller genus. Then for every embedding $f: G \to \Sigma$, each face of $f$ is homeomorphic to an open disc. 
\end{theorem}

% *** ---- *** 
\begin{proof}[Proof of \Tr{thm CS}]
The idea for constructing a universal graph $U_\Sigma$ for $\cc_\Sigma$ is to fix a universal planar graph $P$, consider all possible finite graphs \g embeddable in $\Sigma$, attach a portion of $P$ along each face boundary of \g in each of the countably infinitely many possible ways, and take the disjoint union of the graphs arising from each possibility. Let us make this more precise.

Fix an embedding $f_P: P \to \BS^2$ of a $<$-universal planar graph $P$ as provided by \Tr{thm DK}. %, as constructed by \DK\ \cite{DieKuhUni}. 
For every finite graph \g embeddable into a closed orientable surface of genus at most $\gamma(\Sigma)$, let  $\Sigma'$ be such a surface of minimum genus, and %with $\gamma(\Sigma') \leq \gamma(\Sigma)$, 
fix an embedding $f_G: G \to \Sigma'$. For every cycle $C$ of \g that bounds a face of $f_G$, choose a cycle $C'$ of $P$ with length $|C'|\geq |C|$. (Here, if a face $F$ of $f_G$ is not bounded by a cycle but by a closed walk with repeated vertices, we ignore $F$.) Choose one of the two sides $R$ into which $C'$ separates $P$ in $f_P$ by the Jordan curve theorem, and denote it by $P_{C}$; that is, $P_{C}:= f_P^{-1}(\cls{R})$. Form a minor $P'_{C}$ of $P_{C}$ by contracting some subarcs of $C'$ so that the resulting minor of $C'$ is a cycle $C''$ with $|C''| = |C|$. 
Then form an infinite graph $G'$ as follows. For each facial cycle $C$ of $G$ \wrt\ $f_G$, attach $P'_{C}$ onto $G$ by identifying $C$ with $C''$ using any of the $2 |C|$ possible automorphisms $i: C \to C''$ as the attachment map.

Note that there are countably infinitely many graphs $G'$ that arise this way: \ta\ $\omega$ ways to choose $C'$ for each of the finitely many facial cycles $C$ of \G, after which \ta\ finitely many ways to contract $C'$, and attach $P'_{C}$  onto \G. We claim that each such $G'$ is embeddable into $\Sigma'$, and therefore $\Sigma$. Indeed, each $P'_{C}$ is  planar and hence embeddable into the face $F_C$ of $f_G$ bounded by $C$. Here, we can use \Tr{Youngs} for convenience, to deduce that $F_C$ is homeomorphic with $\R^2$. By uniting $f_G$ with such embeddings of the $P'_{C}$ we obtain an embedding of $G'$ into $\Sigma'$.

Let $U'_G$ be the disjoint union of all these graphs $G'$, and let $U_G = \omega \cdot U'_G$ be the disjoint union of $\omega$ copies of  $U'_G$. Finally, let $U_\Sigma$ be the disjoint union of $U_G$ over all finite \g embeddable into $\Sigma$. By the above claim, we have $U_\Sigma \in \cc_\Sigma$.
\bigskip

It remains to prove that the graph $U_\Sigma$ we just constructed is $<$-universal for $\cc_\Sigma$. For this, let $Y\in \cc_\Sigma$ be connected. Let $X$ be a \lf\ connected graph \st\ $Y< X$, which exists by \Lr{lem lf}. We need to show that $X< U_\Sigma$.

Let $H$ be a finite subgraph of $X$ with $\gamma(H)=\gamma(X)$. Here, we used the well-known fact that a countable graph $X$ embeds into a surface $S$ if each of its finite subgraphs $H$ does, which can be proved along the lines of \Prr{generous}. We can assume for convenience that $H$ is an induced  subgraph of $X$. Let $g: X \to \Sigma'$ be an embedding into the closed orientable surface of genus $\gamma(X)$. Note that $g$ induces an embedding $g_H: H \to \Sigma'$. By \Tr{Youngs}, every face $F$ of $g_H$ is homeomorphic to an open disc. Therefore, the subgraph $X_F:= g^{-1}(F)$ of $X$ embedded into $F$ is planar.  We are hoping to find $X$ as one of the possible copies of $G'$ in our construction of $U_G$ above, but one of the obstacles is that some such faces $F$ may not be bounded by a cycle, meaning they have been ignored during the construction of $G'$, which may then fail to contain a copy of $X_F$.

To address this obstacle, we will \defi{`fatten'} $X$ to obtain a graph $X^\otimes$ embedded in $\Sigma'$ with better-behaved faces. %The idea is to attach a sequence of triangles along each facial walk (\fig{}). This is a very simple construction when $X$ is e.g.\ a finite plane graph, and we now provide the details showing that it can be carried out for \lf\ $X$ embedded in any surface. Formally, we define $F(X)$ as follows\footnote{This construction is repeated from \cite{GeoKon}}. 
We assume that $g$ is generous, which we can by \Prr{generous}, and let $\{D_v \mid v\in V(X)\}$ and $\{D_e \mid e\in E(X)\}$ be collections of discs witnessing the generosity of $g$. We may further assume that the $D_e$ are pairwise disjoint. 

\begin{definition} \label{def otimes}
{\textup{
Let $X''$ be  the multigraph obtained from $X$ by adding two  parallel edges $e',e''$ to each edge $e\in E(X)$. We can extend $g$ to embed $e',e''$ using the fact that the former is generous around edges: we can embed $e',e''$ into the disc $D_e$ so that the circle $e' \cup e''$ separates $e$ from the rest of $X$. Then, for each  $e$ of $X''$ with end-vertices $u,v$, subdivide $e$ into a path of length 3 by declaring two interior points  $e_u,e_v$ of $e$ to be vertices, so that $g(e_u)\in D_u$ and $g(e_u)\in D_v$. Finally, for each vertex $v$ of $X''$, and each two edges $e,f$ incident with $v$ that appear consecutively in $D_v$, add an edge between $e_v$ and $f_v$, and embed it inside $D_v$. We embed these edges in such a way that they form a circle $S_v$ separating $v$ from any other vertex of $X$. Let $X^\otimes$ denote the resulting graph, and $g'$ its embedding into $\Sigma'$; see \fig{figfat}. }}
\end{definition}
Note that $X$ is a minor of $X^\otimes$, obtained by contracting the edges inside each $S_v$ to a point. Thus all we have to do now is to find $X^\otimes$ as a minor of  $U_\Sigma$.

Let $H^\otimes$ be the subgraph of $X^\otimes$ corresponding to $H$, that is, the subgraph of $X^\otimes$ induced by $\{v\cup S_v \mid v\in V(H)\}$. Note that $H^\otimes$  has three types of faces in $g'$: (i) triangles comprising one edge of an $S_v$ and two edges of $v$, (ii) squares joining two $S_v$'s, and (iii) faces corresponding to some face of $g$. Thus by construction, every face $F$ of $H^\otimes$ is now bounded by a cycle $C_F$ of $H^\otimes$, in other words, the closure $\cls{F}$ is homeomorphic with a closed disc of $\Sigma'$. Let $X_F:= g^{'-1}(\cls{F})$ be the subgraph of $X^\otimes$ embedded in $\cls{F}$. Note that $X_F$ is connected, because both $X^\otimes$ and $C_F$ are. Moreover, $X_F$ is planar since it is embedded into the disc $\cls{F}$.

\medskip
We claim that $X^\otimes<U_\Sigma$. %, where the latter is defined in the construction above of $U_\Sigma$. 
Recall that $U_\Sigma$ contains infinitely many graphs of the form $G'$, where $G$ is any finite graph embeddable into $\Sigma'$, and set $G= H^\otimes$.  Establishing the above claim is slightly more complicated than observing that each $X_F$ above is a minor of $P$, because we attached certain minors of $P$ to $G$ to construct $G'$, rather than $P$ itself. %  $P'_{C_F}$, because ... 
%Our task becomes easier in the special case where 
Another convenient feature of our construction of $X^\otimes$ is the following property:
\labtequ{star}{every vertex $v$ of $C_F$ has exactly one incident edge inside $F$ (and two other edges in $C_F$).}
Indeed, our $v$ lies on some $S_v$, and $F$ lies completely inside or outside $S_v$. (In the former case $X_F$ is just the triangle bounding $F$; it is the latter case that is interesting.)
%We will now prove our claim under the assumption that this is the case for every face $F$ of $G$, and we will later show how to reduce the general case to this special one. 

Since $X_F$ is planar, there is a minor $\Delta$ of $P$ isomorphic to $X_F$. By \eqref{star}, each vertex $v$ of $C_F$ has degree at most 3 in $X_F$, and so we may assume that the corresponding branch set $B_v$ of $\Delta$ is a finite tree with exactly three leaves. Let $P_v$ be the subpath of this tree joining its two leaves participating in the copy of $C_F$ in $\Delta$. % (\fig{}). 
 Then $\bigcup_{v\in V(C_F)} P_v$ induces a cycle $C'_F$ in $P$. Note that $f_P(C'_F)$ separates $\BS^2$ into two regions by the Jordan curve theorem, and as $\Delta$ is connected, one of these regions $A$ contains $f_P(\Delta)$, where $f_P$ denotes our fixed embedding of $P$ from above. Recall that when we constructed $G'$, we attached a minor $P'$ of $P$ onto each face boundary $C$ of $G$, and notice that $f_P^{-1}(\cls{A})$ is a candidate for $P'$. Applying this argument to each face $F$ of $G$, we deduce that $X^\otimes$ is a minor of $U_G$, and hence of $U_\Sigma$ as claimed.

\medskip
Thus we have shown that every connected $Y\in \cc_\Sigma$ is a minor of $U_\Sigma$. By repeating the argument on each component, recalling that $U_\Sigma$ contains infinitely many isomorphic copies of each of its components, we obtain the same conclusion for disconnected $Y$ as well. Thus $U_\Sigma$ is a \mue\ of $\cc_\Sigma$.

%It remains to reduce the general case to the case where \ref{star} holds, and for this we will slightly modify $X$ into a graph $X'$ that contains $X$ as a minor and satisfies \ref{star}. For this, Let $P_n$ denote the prism graph with $2n$ vertices, i.e.\ $C_n \times K_2$. For every face $F$ of $G$, recall that $\partial F$ is a cycle $C_F$. Obtain an auxiliary graph $X'_F$ from $X_F$ by adding a new cycle $C'_F$ with $|C'_F| = |C_F|$ disjoint from $X_F$, and joining each vertex of  $C'_F$ with an edge to the corresponding vertex of $C_F$. We call these new edges the \defi{rungs}. Clearly, $X'_F$ is still planar. Finally, for each face $F$ of $G$, attach $X'_F$ onto $G$ by identifying $C'_F$ with the cycle $\partial F$ of $G$, imitating the way in which $X_F$ attaches to $\partial F$ in $X$. Let $X'$ be the resulting graph. We obtain $X$ as  a minor  of $X'$ by contracting each rung. Cleary $X'$ satisfies \ref{star} by construction, and so appealing to the above special case we deduce that $X < X' < U_\Sigma$ as claimed.
\end{proof}

\begin{figure} 
\begin{center}
\includegraphics[width=0.9\linewidth]{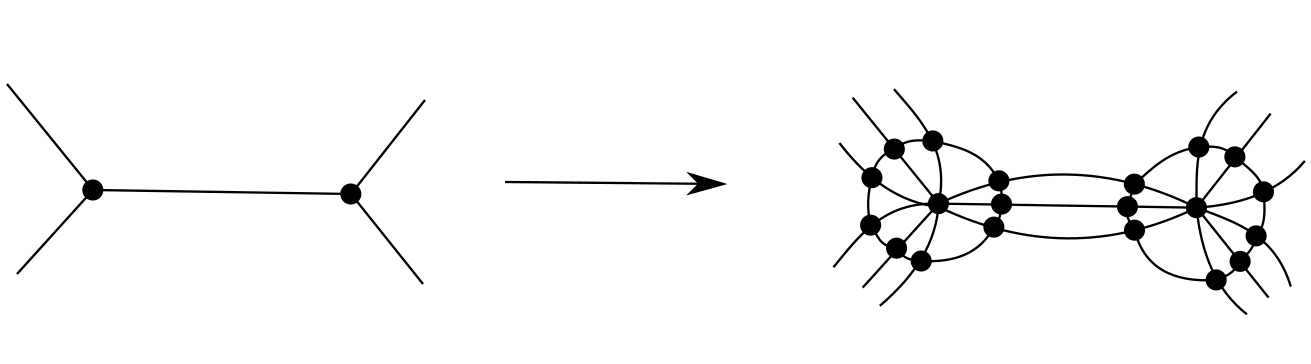} 
\end{center}
\caption{A portion of $X$ (left), and the corresponding part of $X^\otimes$ (right).} \label{figfat}
\end{figure}

\begin{remark}
The reason why we restrict \Tr{thm CS} to orientable $\Sigma$ is that \Tr{Youngs} fails in the non-orientable case; see \cite{PaPiPiVe} for details.
\end{remark}

Let $\Gamma$ denote the class of countable graphs embeddable in some closed orientable surface. In other words, $\Gamma$ comprises the graphs of finite, but unbounded, orientable genus.
Imitating the proof of \Tr{Kinfty}, we have
\begin{proposition}
$\Gamma$ does not have a \mue.
\end{proposition}
% *** ---- *** 
\begin{proof}
Suppose $U$ is a universal graph in $\Gamma$, and let $U':= U \dot{\cup} K_5$. Notice that $U'\in \Gamma$, and therefore $U'< U$. Let $K_0$ be a minor of $K_5$ that extends to a minor of $U'$ in $U$.  Apply the same argument on $U \sm K_0$, and repeat, to find $K:= \omega \cdot K_5$ as a minor of $U$. But this is impossible since $K\not\in \Gamma$, a contradiction proving that $U$ does not exist.
\end{proof}

Let $\Gamma^*$ denote the class of countable graphs 
 every component of which lies in $\Gamma$. Then $\Gamma^*$ does have a \mue, namely the disjoint union of all the universal graphs provided by  \Tr{thm CS}, when $\Sigma$ ranges over all closed orientable surfaces.

%SSSSSSSSSSSS
\subsection{2-compexes}

The proof of \Lr{lem lf} is likely to be adaptable to higher dimensions, in particular to embeddings of 2-complexes into $\R^3$. This may be helpful in addressing the following:

\begin{question} \label{Q 3D}
Does the class of countable 2-complexes embeddable into $\R^3$ (or $\R^4$) have a \mue?
\end{question}

Here, the notion of minor we refer to is that of Carmesin \cite{CarEmbI}, who obtained a characterisation of the finite 2-complexes embeddable into $\R^3$ reminiscent of Kuratowski's theorem. The interested reader is advised to also consult \cite[Theorem 3.2]{GeoKon}. 
 
The analogue of \Qr{Q 3D} can be asked for the sub-family of those 2-complexes embeddable into $\R^3$ without accumulation points of vertices. It might be that a disjoint union of copies of the 3-dimensional cubic lattice is a universal element, in analogy with K\"uhn's result \cite{KuhMin}.

%%%SSSSSSSSSSS
\section{Interlude: every embedded graph has a 3-connected supergraph} \label{sec int}

In this section we introduce another way to `fatten' an embedded graph similar to $X^{\otimes}$ from \Dr{def otimes}, but more convenient for the purposes of the proof of \Tr{thm K5} in the next section.
%we will \defi{`fatten'} $X$ to obtain a graph $F(X)$ embedded in $\Sigma'$ with better-behaved faces. 
The idea is to attach a sequence of triangles along each facial walk (\fig{figFG}). This is a very simple construction when $X$ is e.g.\ a finite plane graph, and we now provide the details showing that it can be carried out for \lf\ $X$ embedded in any surface. Formally, we define our new `fattening' $F(X)$ as follows.
\begin{figure} 
\begin{center}
\includegraphics[width=0.8\linewidth]{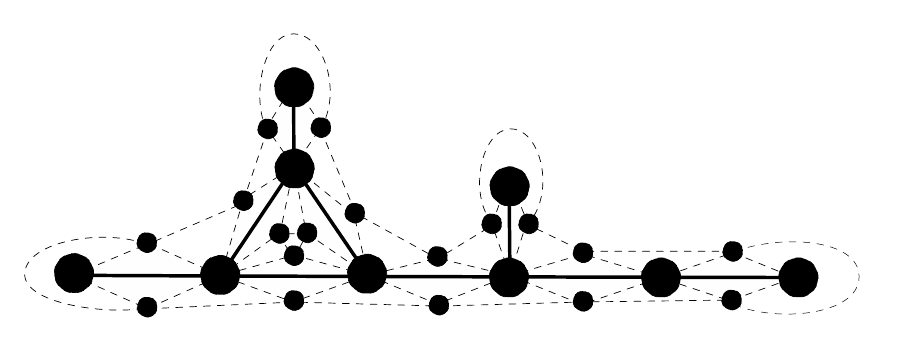} 
\end{center}
\caption{An example of a plane graph \g (solid lines) and its `fattening' $F(G)$ (dashed lines).} \label{figFG}
\end{figure}
\begin{definition} \label{def fat} {\textup{
Let $X$ be a \lf\ graph and  $g: X \to \Sigma$ a  generous embedding into some surface $\Sigma$, the generosity of which is witnessed by collections of discs $\{D_v \mid v\in V(X)\}$ and $\{D_e \mid e\in E(X)\}$.
We define a super-graph $F(X)$ of $X$  embedded in  $\Sigma$ as follows. For each edge $e=uv\in E(X)$, we embed a square $Q_e$ with vertices $u t_1 v t_2$ into $D_e$, where $t_1,t_2$ are new vertices. We embed $Q_e$ so that $e$ lies in its interior. Thus $D_e$ now contains two triangles $T_1=u t_1 v, T_2=u t_2 v$ sharing $e$. We call $t_i$ the \defi{tip} of $T_i$. Having done so \fe\ $e=uv\in E(X)$, we introduce some further edges between tips of adjacent $T_i$'s as follows. 
For every $v\in V(X)$, and any two tips $t,s$ \st\ the edges $vt,vs$ are consecutive in our embedding, we embed a $ts$ edge into $\Sigma$. We thereby make sure that the image of $ts$ does not intersect any other edge, new or old. (To see that this is possible, extend $D_v$ to a topological disc $D'_v$ that contains all tips adjacent to $v$, and embed each $ts$ inside $D'_v$.)
Let $F(X)$ denote the resulting graph, and $g'$ its embedding into $\Sigma$; see \fig{figFG}}. 
}
\end{definition}
We remark that for each (finite or infinite) facial walk $W=  \ldots e_0 e_1 \ldots$ bounding a face $A$ of $X$ in $g$, there is a facial path of $F(X)$ in $g'$ obtained by replacing each $e_i$ by the tip of its triangle embedded in $A$. 

\medskip
We make the following observation about $F(X)$ that may be of independent interest. It simplifies and generalises a construction of \cite{Kleinian} (Lemma 5.3).
%%%%%%%%%%
\begin{proposition} \label{pr fat}
Let $G \subset \Sigma$ be a connected \lf\ graph embedded in any surface. Then $F(G)$ is 3-connected.
\end{proposition}
We emphasize that \g is simple; the result fails in the presence of either loops or parallel edges. We allow $\Sigma$ to be non-orientable and non-compact.
%%%%%%%%%%
\begin{proof}
For $v\in V(G)$, let $C_v$ denote the subgraph of $F(G)$ induced by the neighbours of $v$. Note that $C_v$ has a hamilton cycle, obtained by visiting the neighbours of $v$ in the order their edges are embedded around $v$ in $\Sigma$. Thus the subgraph $G_v$ of $F(G)$ induced by $C_v \cup \{v\}$ has a spanning \defi{wheel}, and in particular $G_v$ is 3-connected. ($G_v$ is isomorphic to a wheel when $v$ lies in no triangle of $G$.)
 
For $e=vu\in E(G)$, the \defi{closed neighbourhood} $N_e:= G_v \cup G_u$ of $e$ contains the union of two wheels sharing at least 4 vertices (namely $u,v$, and the two tips of $e$). Thus $G_v$ cannot be separated from $G_u$ by any two vertices. Since \g is connected, it follows that $G= \bigcup_{v\in V(G)} G_v$ is 3-connected.
\end{proof}
%%%%%%%%%%

%%%SSSSSSSSSSS
\section{A universal $K_5$-free (and $K_{3,3}$-free) graph} \label{sec K5}

In this section we prove the most difficult theorem of this paper:

\begin{theorem} \label{thm K5}
The class of countable graphs with no $K_5$ minor has a \mue.
\end{theorem}

A classical theorem of Wagner \cite{Wagner} says that a finite graph \g is $K_5$-minor-free if and only if \g can be built from planar graphs and the \defi{Wagner graph} by repeated 3-clique sums, where the Wagner graph $W$ is the one depicted in \fig{Wagnergraph}. 

\begin{figure} 
\begin{center}
\includegraphics[width=0.2\linewidth]{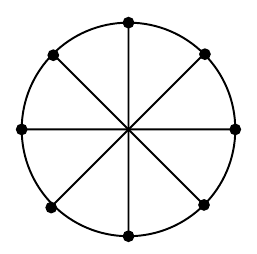} 
\end{center}
\caption{The Wagner graph.} \label{Wagnergraph}
\end{figure}

\kriz\ \& Thomas \cite{KriThoCli} (\Tr{kriz} below) extended Wagner's result to infinite graphs. Our construction of a universal graph $U$ proving \Tr{thm CS} will use this decomposition: the idea is to construct $U$ by repeated 3-clique sums of a universal planar graph $P$ and copies of $W$.  However, this will not work if we let $P$ be any universal planar graph. Suppose, for example, $H$ is a 3-sum of two planar graphs $G_1,G_2$, and $K \subset G_1 \cap G_2$ is a common triangle. Although we can find both $G_1,G_2$ as minors of $P$, it might be that $H$ is not a minor of any 3-sum of two copies of $P$, just because $P$ may have no triangles. Thus we need to modify $P$ to ensure that it does have a triangle wherever one is needed. This may look hopeless at first sight, but it turns out there is a surprisingly simple construction that achieves it ---although proving that it does requires some effort. Once we have constructed $P$ carefully enough that any $H$ as above is a minor of some 3-sum of two copies of $P$, it will not be hard to show that any $K_5$-minor-free graph is a minor of a certain repeated 3-sum of infinitely many copies of $P$ and $W$.

We start our construction of $P$ by letting $P_0$ be a universal planar graph, e.g.\ the one of \DK\ \cite{DieKuhUni}. By \Lr{lem lf}, we may assume that $P_0$ is \lf. Pick any embedding $f_0: P_0 \to \BS^2$, and let $P=F(P_0)$ be the corresponding fattening. Thus $P$ is a \lf\ universal planar graph with plenty of triangles. We think of $P$ as a plane graph with an embedding $f_P$ induced by $f_0$ (in fact $P$ is 3-connected by \Prr{pr fat}, so its embedding is known to be essentially unique \cite{ImWhi,Kleinian}, but we will not use this fact). As mentioned above, we need something stronger than the existence of triangles in $P$; we will need the following property \eqref{triang}. %Given graphs $H,G$, an $H$-minor $\Delta$ in $G$ is said to be \defi{minimal}, if each of its branch sets $B_v$ is a tree with $k$ leaves, where $k=d_H(v)$ Thus a minimal $K_3$-minor is a cycle $C$ with $V(C)$ decomposed into three subpaths
A \defi{minimal $K_3$ minor}  of a graph $P$ is  
a family $\Delta$ of three disjoint (possibly trivial) paths $M_1,M_2,M_3 \subset P$, called the branch sets of $\Delta$, and three edges $e_1,e_2,e_3$ joining the $M_i$'s so that $\bigcup_{1\leq i \leq 3} (M_i \cup e_i)$ is a cycle $C(\Delta)$.
%A $K_3$-minor $\Delta$ of a graph $P$ is said to be \defi{minimal}, if the union of its three branch sets with a selection of three branch edges joining them is a cycle $C(\Delta)$ of $P$. In particular, each branch set of $\Delta$ is a (possibly trivial) path. Note that the branch sets of $\Delta$ determine $C(\Delta)$. %Thus the union of the branch sets and branch edges of such a $\Delta$ is a cycle $C(\Delta)$ of $P$. %\mymargin{$P=F(P_0)$ here}
%
\labtequ{triang}{For every minimal $K_3$ minor $\Delta$ of $P_0$, and any of the two regions $R$ into  which $C(\Delta)$ separates $\BS^2$, there is a triangle $\Delta'\subset P=F(P_0)$ contained in $\cls{R}$, and three disjoint paths $P_i, 1\leq i\leq 3$ joining each vertex of $\Delta'$ to a distinct brach set of $\Delta$ (\fig{figD}).}
\begin{figure} 
\begin{center}
\begin{overpic}[width=.3\linewidth]{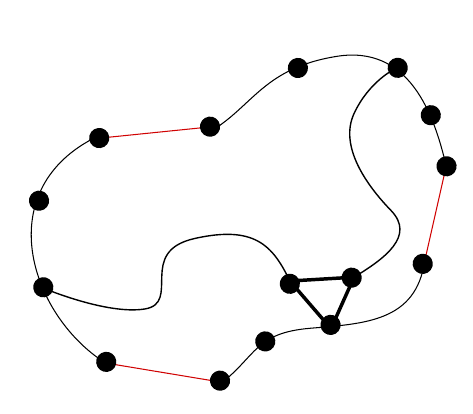} 
\put(45,40){$P_2$}
\put(70,8){$P_3$}
\put(65,53){$P_1$}
\put(63,30){$\Delta'$}
\put(31,2){$e_1$}
\put(94,37){$e_3$}
\put(28,64){$e_2$}
\end{overpic}
\end{center}
\caption{A minimal $K_3$ minor (exterior cycle) and the triangle $\Delta'$ (bold lines) as in \ref{triang}.} \label{figD}
\end{figure}
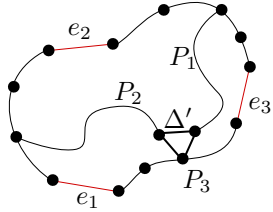
We emphasize that $\Delta'$ is a subgraph of $P$, not just a $K_3$ minor. Before we prove \eqref{triang}, let us see how it helps. Recall our above example $H$, namely a 3-sum of two planar graphs $G_1,G_2$ along a common triangle $K$. As a warm-up towards the proof of \Tr{thm K5}, we will show that $H$ is a minor of a 3-sum of two copies $P_1,P_2$ of $P$. To see this, let $H_i$ be a $G_i$ minor of $P_i$ for $i=1,2$. Let $K_i$ be the part of $H_i$ corresponding to $K$. Pick a subpath of each of the three branch sets of $K_i$ to form a minimal $K_3$ minor $\Delta_i$ of $P_i$. Assume for simplicity that $K$ bounds a face in any embedding of $G_i$ (this can be achieved, because if $K$ separates $G_1$, then we can further decompose $G_1$ as a 3-sum). Then we can choose the $\Delta_i$  so that one of the two sides $R_i$ of $C(\Delta_i)$ does not intersect the rest of $H_i$. Apply \eqref{triang} to each triple $\Delta_i, P_i, R_i$ to obtain two triangles $\Delta_i'\subset P_i$ and their corresponding paths $P_{ij}, i\in \{1,2\},  j\in \{1,2,3\}$. Identify $\Delta_1'$ with $\Delta_2'$ to obtain a 3-sum $Q$ of $P_1,P_2$. (To picture this, imagine a copy of \fig{figD} as a substructure of each $P_i$, with the two copies of $\Delta'$ glued together.) Extend $H_1$ by attaching one of the $P_{1j}$ to each of the three branch sets of $K_1$. Do the same for $H_2$. By uniting  $H_1$ with  $H_2$ we thus find a minor of $Q$ isomorphic to $H$.

\begin{proof}[Proof of \eqref{triang}]
Given a cycle $C$ of a plane graph \G, and one of the sides $A$ into which $C$ separates the plane, we say that $C$ is \defi{induced inside $A$}, if no chord of $C$ lies in $A$. 

Let us first consider the case where $C=C(\Delta)$ is induced inside $R$. Let $B$ be one of the branch sets of $\Delta$ ---hence $B$ is a sub-path of $C$--- and let $e=uv$ be the edge of $C$ joining the other two branch sets of $\Delta$ (\fig{figQ}). Let $t$ be the tip of $e$ inside $R$. The desired triangle $\Delta'$ is $uvt$. Two of the desired paths $P_1,P_2$ are the singletons $\{u\}$ and $\{v\}$, and so it remains to find a $t$--$B$~path $P_3$ in $R$ avoiding $u,v$. For this, pick a vertex $x$ in $B$, let $Q$ be one of the two $x$--$e$~paths in $C$ and let $t'$ be the tip of the first edge of $Q$. Recall from \Prr{pr fat} that the closed neighbourhood $G_w$  of every $w\in V(P_0)$ has a spanning wheel, and note that $G_w \cap R$ contains a path $Q_w$ joining the tips of the two edges of $w$ in $C$. By concatenating the $Q_w$ as $w$ moves along $Q$ we obtain a $t$--$t'$~path in $R$, and by appending the edge $t'x$ we obtain the desired $t$--$x$~path $P_3$.
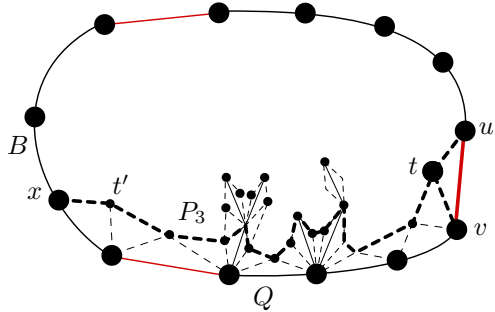
\begin{figure} 
\begin{center}
\begin{overpic}[width=.55\linewidth]{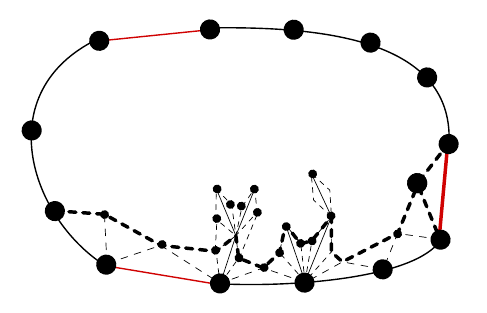} 
\put(1,30){$B$}
\put(35,17){$P_3$}
\put(94,14){$v$}
\put(95,34){$u$}
\put(81,26){$t$}
\put(22,22){$t'$}
\put(5,21){$x$}
\put(50,0){$Q$}
\end{overpic}
\end{center}
\caption{Finding $P_3$ in the proof of \eqref{triang}.} \label{figQ}
\end{figure}
It remains to handle the general case where $C=C(\Delta)$ is not induced inside $R$, but it is possible to reduce it to the induced case as follows.
Given a plane graph $G$, a minimal $K_3$ minor $\Delta_0$ of $G$, one of the two regions $A$ into which $C(\Delta_0)$ separates $\BS^2$, and a chord $h$ of $C(\Delta_0)$ in $A$, we can use $h$ as a shortcut of $\Delta$ to replace it by a minimal $K_3$ minor $\Delta_1$ of $G$ inside $A$ each branch set of which is contained in a distinct branch set of $\Delta_0$, and at least one of these containments is proper. We repeat as long as possible, ending up with a $\Delta_k$ which is induced inside a subregion $R$ of $A$. Thus replacing $\Delta$ with $\Delta_k$ we can reduce to the induced case handled above. Note that each branch set of $\Delta_k$ is contained in a distinct branch set of $\Delta=\Delta_0$ by induction, and so the $P_i$'s we obtain do reach $\Delta$.
\end{proof}

Our universal graph will be constructed so that it comes with a tree decomposition into parts each of which is a copy of either $P$ or the Wagner graph $W$. Let us recall the relevant definitions.

\medskip 

A  \defi{tree decomposition} of a graph \G\ is a pair $(T, X)$, where $T$ is a tree and $X = \{X^t \mid t\in V(T) \}$ is a family of subgraphs of \G, called \defi{torsos}, such that
\begin{itemize}
	\item[(T1)] \label{TD i}  $\bigcup_{t\in V(T)} X^t= V(G)$, 
	\item[(T2)] \label{TD ii}  every edge $e$ of \G\ lies in some $X^t$, and
	\item[(T3)] \label{TD iii} if $t, s, q \in V(T)$ and $s$ is on the path between $s$ and $q$, then $X^t \cap X^{q} \subseteq X^{s}$.\\

We say that a tree-decomposition $(T, X)$ is a \defi{$k$-decomposition} of $G$ over a graph class $\Gamma$, if 
	\item[(T4)] \label{TD iv} $|V(X^t \cap X^{s})| \leq k$ for every ${ts} \in E(T)$, and,\\ 
	 
letting $G^*$ be the graph obtained from \G\ by joining every edge which has both its endpoints in some $X^t \cap X^{s}$ for some ${ts} \in E(T)$,
	\item[(T5)] \label{TD v} $G^* \restr V(X^t) \in \Gamma$ for every $t \in V(T)$.
\end{itemize}

Let $G, G'$ be two graphs, and let  $K \subseteq G, K'\subseteq G'$ be cliques of equal size. A \defi{clique-sum} of $G, G'$ is a graph $H$ obtained from their disjoint union by identifying $K$ with $K'$ via a bijection of their vertices, and then possibly deleting any of the edges of the resulting clique $K=K'$ of $H$. A \defi{$k$-clique-sum}, or just \defi{$k$-sum}, is a clique-sum in which these cliques have at most $k$ vertices.

A graph \g is called a \defi{$k$-simplex} if it does not contain a $k$-cut, i.e.\ a clique on at most $k$ vertices the removal of which disconnects \G. 

Let $\Gamma$ denote the class of countable graphs that are either planar or isomorphic to $W$.
\medskip

Our universal graph $U$ will come with a $3$-decomposition $(S,Y)$ over $\Gamma$, where $S$ is the regular tree of degree $\omega$, and each $Y^t$ is isomorphic to either $P$ or $W$. We will define $U$ in such a way that makes it maximal with these properties in the sense that any other such graph is a subgraph of our $U$, although we will not formally state or use this maximality property. 

Our construction of $U$ follows the lines of Mohar's tree amalgamation \cite{MohTre}, but with certain differences that make it more restricted in some sense and more general in another. The idea is to keep forming 3-sums of infinitely many copies of $P$ and $W$ in all possible ways. Our formal definition of $U$ is as follows.

% each adhesion set is a clique of size at most 3, and for each $Y^t$ and each clique $K\subset Y^t$ with $|K|\leq 3$\footnote{Since $W$ has no triangles, we can take  $|K|\leq 2$ in case $Y^t \isom W$.}, there are infinitely many neighbours $s$ of $t$ in $S$ \st\ $Y^s \isom P$ and $Y^s \cap Y^t= K=K'$ \fe\ clique $K'\subset Y^s$ with $|K'| = |K|$, and the same with $W$ instead of $P$ Constructible using Mohar's tree amalgamation

\begin{definition} \label{def U}
\textup{Let $S$ be a tree every vertex of which has countably infinite degree (this fixes $S$ up to isomorphism). Let $k$ be a colouring of $V(S)$ by two colours, red and blue say, \st\ every vertex has infinitely many neighbours of each colour. For each $v\in V(S)$, we let $Y^v$ be a copy of $P$ if $k(v)$ is blue, and a copy of $W$ if $k(v)$ is red. Moreover, we label each directed edge $uv$ of $S$ by an isomorphism $c_{uv}$ from some clique $K$ of $Y^u$ with at most 3 vertices to a clique of $Y^v$. Thus $K$ is a triangle, edge, or vertex, and if  $Y^u$ or $Y^v$ is a copy of $W$, it must be an edge or vertex. We choose the labelling $c$ in such a way that for every $u\in V(S)$, every clique $K$ of $Y^u$ with at most 3 vertices, and every isomorphism $h$ from $K$ to a subgraph of $P$ or $W$, there are infinitely many outgoing edges $uv$ \st\ $c_{uv}=h$ (whereby $v$ is blue if $h(K)$ is a subgraph of $P$, and red otherwise). It is easy to construct such a labelling greedily, starting from a `root' vertex of $S$.}

\textup{To define $U$, we take the disjoint union of $\{Y^v \mid v\in V(T)\}$, and then for every directed edge $uv$ of $S$, we identify each vertex $x$ in the domain of $c_{uv}$ with $c_{uv}(x)$.
}
\end{definition}

This completes the construction of $U$. Clearly, $(S,Y)$ is a 3-decomposition of $U$. To establish \Tr{thm K5} it now remains to prove

\begin{proposition} \label{prop U}
For every countable graph \g with no $K_5$ minor, $G< U$ holds.
\end{proposition}

Before proving this, we state the aforementioned result of \kriz\ \& Thomas.

\renewcommand{\theenumi}{(\Alph{enumi})}
\begin{theorem}[{\cite[Theorems~3.5~and~3.9]{KriThoCli}}] \label{kriz}
A countable graph \g is $K_5$-minor-free  if and only if \g has a supergraph $G^*$ which has a tree decomposition $(T,X)$ \st\ 
\begin{enumerate}
	\item \label{T i} $G[X^t]$ is a 3-simplex in $\Gamma$ \fe\ $t\in V(T)$, and
	\item \label{T ii} $X^s \cap X^t$ is a clique on at most 3 vertices \fe\ $s,t\in V(T)$.
\end{enumerate}
\end{theorem}

\begin{remark}
\Tr{kriz} is not formulated in \cite{KriThoCli} exactly as stated here, so let me explain how our formulation follows from the results of \cite{KriThoCli}. A \defi{3-complex} over $\Gamma$ is a graph \g such that every 3-simplex contained in \g belongs to $\Gamma$. By \cite[Theorem~3.9]{KriThoCli}, if every finite subgraph of \g is a subgraph of a 3-complex over $\Gamma$ ---which is the case by Wagner's aforementioned theorem--- then so is \G. %Let $G^* \supseteq G$ be a 3-complex over $\Gamma$. 
Implication $(ii) \to (iii)$ of \cite[Theorem~3.5]{KriThoCli} then says that $G$ admits a tree decomposition $(T,X)$ over $\Gamma$ with adhesion sets of size at most 3 (condition (T4)), and that after making $X^s \cap X^t$ complete both $X^s, X^t$ remain in $\Gamma$ (condition (T5)). We let $G^*$ be the graph obtained from \g after making each adhesion set of  $(T,X)$ complete, and so \ref{T ii} is satisfied. Condition \ref{T i} is satisfied too by the construction of $(T,X)$: each $X^t$ is induced by a maximal clique of an auxiliary graph $H^*$, and condition $(*)$ of  \cite{KriThoCli} says that it is a 3-simplex of $G$.
\end{remark}

\begin{proof}[Proof of \Prr{prop U}]
Apply \Tr{kriz} to obtain the tree decomposition $(T,X)$. We may assume \obda\ that $G=G^*$, because the latter is $K_5$-minor-free too as a consequence of this decomposition. In particular, we may assume that \g is connected. %, although we do not have to use this fact).

Let $\seq{t}$ be an enumeration of $V(T)$ \st\ $T[\{t_0, \ldots, t_k\}]$ is connected \fe\ $k\in \N$. We will inductively embed each $X_i:= X^{t_i}$ as a minor $H_i$ of some torso $Y_i=Y^{s_i}$ of the above tree decomposition $(S,Y)$ of $U$, using \ref{triang} to ensure that adhesion sets are embedded as common sub-minors. The branch sets of each 
$H_i$ will be either fixed or enlarged in subsequent steps, and so we will obtain a minor of \g in $U$ in the limit $i\to \infty$. 

The first step $i=0$ is easy: pick a torso $Y_0$ of $(S,Y)$ which is isomorphic to $W$ if $X_0$ is, and isomorphic to $P$ if $X_0$ is planar. Let $H_0$ be a minor of $Y_0$ isomorphic to $X_0$; in the former case this is trivial, and in the latter it is possible because $P_0$ is a universal planar graph. We insist that $H_0$ be a minor of the $P_0$ subgraph of $Y_0$ in this case, i.e.\ $H_0$ uses no tips of $P=F(P_0)$.

For the inductive step, assume that step $i-1$ concluded with a minor of $\bigcup_{k< i} X_k$ in $U$, contained in the union of finitely many torsos of $(S,Y)$, whereby we have defined a minor $H_k$ of some $Y_k$ isomorphic to $X_k$ \fe\ $k<i$. Let $t^j$ be the unique neighbour of $t^i$ in $T$ \st\ $j<i$, the existence and uniqueness of which is guaranteed by the choice of \seq{t}. As above, we will choose a torso $Y_i=Y^{s_i}$  of $(S,Y)$ which is isomorphic to $W$ if $X_i$ is, and isomorphic to $P$ if $X_i$ is planar, and we will find  a minor $H_i$ of $Y_i$ isomorphic to $X_i$. But this time we have to make our choices carefully. 

We first consider the case where $X_i$ is planar, and $X_i \cap X_j$ is a triangle  $\Delta^o$. This case will be handled almost identically with the warm-up following \eqref{triang}. Let $H_i$ be a minor of $P_0\subset P$ isomorphic to $X_i$. Let $\Delta$ be the  triangle of $H_i$ corresponding to $\Delta^o$.  Recall that $X_i$ is a 3-simplex by \ref{T i}, and therefore $\Delta^o$ is facial in any embedding of $X_i$ into $\BS^2$. This implies that one of the regions $R$ into which $H_i$ separates  $\BS^2$ in our fixed embedding $f_P: P \to \BS^2$ is bounded by $\Delta$ and meets no other branch sets or branch-edges of $H_i$.   Thus, after choosing a subpath of each branch set of $\Delta$ to turn it into a minimal $K_3$ minor of $P$, we can apply \eqref{triang} to $\Delta$ and $R$ to obtain a triangle $\Delta'_i \subset P$ and disjoint paths $P_i, 1\leq i\leq 3$ joining its vertices to distinct brach sets of $\Delta$. Note that the $P_i$ are disjoint from any branch set and edge of $H_i$ not contained in $\Delta$, because they lie in $R$. We modify $H_i$ by enlarging each branch set of $\Delta$ by uniting it with its unique incident $P_i$. Now $H_i$ is still isomorphic to $X_i$, and each branch set of $\Delta$ contains a unique vertex of $\Delta'$.

Notice that in this case, i.e.\ when $X_i \cap X_j$ is a triangle, $X_j$ is planar too because $W$ has no triangles. Thus in step $j$, we defined an $X_j$ minor $H_j$ of some $Y_j \isom P$. By repeating the above arguments with $X_i$ replaced by $X_j$, we can find a triangle $\Delta'_{ji}\subset Y_j$, and enlarge the branch sets of a triangle $\Delta_{ji}$ of  $H_j$ if needed,  to ensure that they each contain exactly one vertex of $\Delta'_{ji}$. Note that it might be that $\Delta_{ji}=\Delta_{jk}$ for some $k<i$, in which case we choose $\Delta'_{ji}=\Delta'_{jk}$, and do not enlarge any branch sets.

We choose $Y_i$ to be a neighbour of $Y_j$ in $S$ ---more precisely, $s_i$ is a neighbour of $s_j$--- that has not been chosen as $Y_k$ in any of the previous steps $k$. There are infinitely many such $Y_i$ to choose from by the construction of $U$. In addition, we make this choice so that the attachment map dictated by the label $c(s^j s^i)$ attaches  
$Y_j$ to $Y_i$ by identifying $\Delta'$ with $\Delta'_{ji}$ so that two vertices get identified whenever they correspond to the same vertex of $\Delta^o$. %; more precisely, we choose an isomorphism  $\iota^i: Y_i \to P$, and a copy $Y_i$ of $P$ in $(S,Y)$ that attaches $\iota^i(\Delta^o)$ to $\iota^j(\Delta^o)$, where $\iota^j$ had similarly been defined at step $j$, or otherwise we define it now to be any isomorphism from . 
Then, for each vertex $v$ of $\Delta^o$, we unite the corresponding branch sets $B_v^j,B_v^i$ of $H_j$ and $H_i$ into one branch set $B$, noting that $B$ is connected because  $B_v^j,B_v^i$ contain a common vertex in $\Delta'_{ji}$. We have thus extended our minor of $\bigcup_{k< i} X_k$ in $U$ inherited from step $i-1$ into a minor of $\bigcup_{k\leq i} X_k$. This concludes step $i$ in the case where $X_i$ is planar, and $X_i \cap X_j$ is a triangle.
\medskip

The other cases are simpler. If $X_i$ is planar, and $X_i \cap X_j$ is an edge $e=uv\in E(G)$, then let again $H_i$ be a minor of $P_0 \subset P$ isomorphic to $X_i$. Let $B^i_u,B^i_v$ be the branch sets of $H_i$, corresponding to $u$ and $v$, and pick a $B^i_u$--$B^i_v$~edge $e^i$ in $P$. Likewise, we let $B^{j}_u,B^{j}_v$ be the corresponding branch sets of $H_j$, and pick an edge $e^{j}$ joining them in $Y_j$, but if a $B^j_u$--$B^j_v$~edge $e^{k}$ was chosen in some previous step $k$ (because some $X_k$ also attaches to $X_j$ along $e^k$), then we let  $e^{j}=e^k$.  As above, pick an isomorphism $\iota^i: X_i \to P$, and choose $Y_i$ to be a neighbour of $Y_j$ in $S$ that has not been chosen as $Y_k$ in any of the previous steps $k$, and attaches to $Y_j$ by identifying $e^{i}$ with $e^j$ in the direction dictated by $\iota^i$. For each end-vertex $x$ of $e$, we unite $B^j_x$ and $B^i_x$ into one set, which set we declare to be our $i$th branch set of $x$.
\medskip

If $X_i$ is a copy of $W$, and $X_i \cap X_j$ is an edge, we proceed as in the previous case, with the obvious difference being that we choose $Y_i$ to be a copy of $W$ instead of $P$. 

Finally, if $X_i \cap X_j$ is a vertex $v\in V(G)$, then our task becomes easier. We proceed as above, and ensure that $Y_i$ attaches to $Y_j$ so that the branch sets corresponding to $v$ in $H_i$ and $H_j$ contain the unique common vertex of $Y_i,Y_j$, and we unite these two branch sets into one.
\medskip

Note that at each step $i$ we have associated to each $v\in V(\bigcup_{k\leq i} X_k)$ a disjoint branch set $B_v^i\subset U$, so that if $vu\in E(\bigcup_{k\leq i} X_k)$ then there is a $B_v^i$--$B_u^i$~edge in $U$. Moreover, $B_v^i$ is increasing with $i$. Thus the family of branch sets $\{\bigcup_{i\in \N} B_v^i \mid v\in V(G)\}$ defines a $G$ minor in $U$.
\end{proof}

Similarly to \Tr{kriz}, \kriz\ \& Thomas \cite{KriThoCli} extended another result of Wagner, saying that the $K_{3,3}$-minor-free graphs are exactly the 2-clique-sums of planar graphs and $K_5$. By imitating the proof of \Tr{thm K5} we can thus deduce

\begin{theorem} \label{thm K33}
The class of countable graphs with no $K_{3,3}$ minor has a \mue.
\end{theorem}

A universal element can be obtained as in \Dr{def U}, except that we replace $W$ by $K_5$, and do not allow $K$ to be a triangle. (We may use $P_0$ instead of $P$ this time.)

\subsection{Bounding the maximum degree} \label{sec Delta}

Recall that every graph \g embeddable in a surface $\Sigma$ is a minor of a graph $H$ of maximum degree $\Delta(H)=3$ that also embeds in $\Sigma$. This property is not enjoyed by all minor-closed classes even if we increase the $\Delta(H)$ we allow. The class of stars $\forb{P_3}$ is an easy example, and one can prove that apex graphs provide another example. 

\begin{definition} \label{def D}
Given a minor-closed graph family $\cc$, define $\Delta(\cc)$ as the smallest cardinal $k$ \st\ \fe\ $G\in \cc$ \ti\ $H\in \cc$ with $\Delta(H)\leq k$ and $G<H$.
\end{definition}

As mentioned above, $\Delta(\cc)$ can be infinite. I am not aware of any studies of this parameter even for finite graphs. For which minor-closed classes  $\cc$  of finite graphs is $\Delta(\cc)$ finite? This question is perhaps too general to admit a concrete answer, but here is a more specific problem that I find interesting: %For example, one could ask

\begin{conjecture} \label{conj Delta}
Every proper minor-closed  class of countable graphs is contained in such a  class \cc\ with finite $\Delta(\cc)$.
\end{conjecture}

\defi{Proper} here means that \cc\ does not contain all countable graphs.

Perhaps the finiteness of $\Delta(\cc)$ can be used to formalise the idea that we do not need apex vertices to describe the graphs in \cc\ in structure theorems in the spirit of Robertson \& Seymour \cite{GM17}. %Bounding $\Delta(\cc)$ might offer a way to 

\medskip
In the rest of this section we prove that $\Delta(\cc)$ is finite for the $K_5$-minor-free and $K_{3,3}$-minor-free graphs.

\begin{proposition} \label{prr 22}
The universal $K_5$-minor-free graph of \Tr{thm K5} can be chosen to have maximum degree 22.
\end{proposition}
% *** ---- *** 
\begin{proof}
We will alter the construction of $U$ in \Dr{def U} so as to never glue two copies of $P$ or $W$ directly; instead, we expand $P$ and $W$ by adding new cliques that we use to form 3-sums without increasing vertex degrees too much. 

\medskip
Let us make this precise. A \defi{$k$-pilar} is a graph of the form $K_k \times R$ where $\times$ stands for the cartesian product, and $R$ is a ray. Its \defi{base} is the subgraph $K_k \times \{r\}$ where $r$ is the first vertex of $R$.

Let $P^\square$ be the graph obtained from $P$ by introducing a  $k$-pilar $R_K$ for every clique $K$ of $P$ of size at most $3$, identifying the base of $R_K$ with $K$. 

Recall that $P$ was defined as $F(P_0)$, where $P_0$ is any universal planar graph. By \Lr{lem lf}, we may assume $\Delta(P_0)=3$. This implies that  $\Delta(P)=9$ by the construction of $F(P_0)$. Moreover, by subdividing each edge of $P_0$ if needed, we may assume that $P_0$ has no $K_3$ subgraph. We claim that 
\labtequ{Psq}{$\Delta(P^\square)\leq 28$.}
Indeed, each pilar we attach to a vertex $v$ of $P$ increases its degree by 1. Note that $v$ is incident with exactly one $1$-pilar of $P^\square$, and at most $d(v)\leq 9$ $2$-pilars. Let us bound the number $3$-pilars.
Since $P_0$ has no $K_3$ subgraph, the closed neighbourhood of $v$ in $P$ is isomorphic to a wheel by construction (recall \Prr{pr fat}), and therefore $v$ lies in exactly 9 triangles of $P$.   Thus $v$ is incident with at most nine $3$-pilars. This gives $d_{P^\square}(v)\leq 9+1+9+9=28$ as claimed. 

 %Since the closed neighbourhood $N_v$ of $v$ in $P$ is planar, and it has at most 10 vertices, $N_v$ has at most 24 edges by Euler's formula. Of these edges at most $24-9=15$ are not incident with $v$, and there is a one-to-one correspondence between those edges and the triangles containing $v$. Thus $v$ is incident with at most 15 $3$-pilars. This gives $d_{P^\square}(v)\leq 9+1+9+15=34$ as claimed. 

We define $W^\square$ similarly to $P^\square$, starting with $W$ instead of $P$. Easily, $\Delta(W^\square)< 28$.

We change the construction of $U$ in \Dr{def U} in two ways. Firstly, instead of letting each torso $Y_i$ of $S$ to be a copy of $P$ or $W$, we use  $P^\square$ or  $W^\square$. Secondly, whenever we previously attached a clique $K$ of $Y_i$ to a clique $K'$ of $Y_j$, we now attach a copy of $K$ from its pilar to a copy from the pilar of $K'$, whereby we do not allow any of these copies to be the base of the pilar. There are infinitely many copies to choose from, and we make those choice in such a way that each copy is only used once. This can easily be achieved e.g.\ by using a bijection between $V(S) \times V(S)$ and the vertices of the ray $R$ used in the definition of the pilars. The resulting graph $U^\square$ thus satisfies $\Delta(U^\square)\leq 28$ by construction, the bottleneck being \eqref{Psq}. 

It remains to check that $U^\square$ is $K_5$-minor-free universal, and this is true because $U< U^\square$. Indeed, by contracting each copy of $R$ in each pilar of our construction we obtain $U$ as a minor of $U^\square$.

\medskip
In fact, a closer inspection of the proof of \Prr{prop U} reveals that we can save a little. It suffices to only form 2-sums of original edges of $P_0 \subset P$ in the construction of $U$. Thus we only need to attach 2-pilars to those edges, saving a factor of 6.
\end{proof}

A universal $K_{3,3}$ graph with bounded maximum degree can be constructing along the same lines. Since we do not need 3-sums in this case, we can define $U^\square$ starting with $P=P_0$, and only attaching 1- and 2-pilars. The resulting upper bound on the degrees of $U^\square$ is $4+1+4= 9$, the bottleneck being $K_5^\square$ rather than $P^\square$.

Thus we have proved \Cr{cor K5}.

\begin{question}
What are the optimal upper bounds on the maximum degree in \Cr{cor K5}?
\end{question}

%SSSSSSSSSSSSSSS
\section{Further remarks and questions} \label{sec Qs} 

\subsection{Forbidding some finite graph componentwise} \label{sec Kn} 

Diestel \& K\"uhn \cite{DieKuhUni} asked whether $Forb(K_n)$ has a \mue. The cases $5<n<\infty$ remain open, and our proof of \Tr{thm K5} suggests that they may be hard to answer. Let us suppose that $Forb(K_n)$ has a \mue\ $F_n$ \fe\ $n$. Then $\bigcup_{n\in \N} F_n$ would be universal for the class \cf\ of graphs \g \st\ no component of \g contains all finite graphs as minors. It would be interesting to try to prove this directly:

\begin{conjecture} \label{conj cf}
The class \cf\ has a \mue.
\end{conjecture}

It is not hard to prove that $\cf= \forb{X}$ where $X$ is the graph obtained from $\{K_n \mid \nin\}$ by picking one vertex from each $K_n$ and identifying them. %cone over the disjoint union $\bigcup_{\nin} K_n$, or in other words, the join of $\

To prove \Cnr{conj cf} it would suffice to find a sequence \seq{\cc} of minor-closed classes with \mue s $U_n$ \st\ every finite graph belongs to some $\cc_n$ using the above idea. Candidates for such sequences can be sought among geometrically defined classes of graphs. One example is letting $\cc_n$ be the class of graphs of `dimension' $d=n$, where $d$ is e.g.\ the Colin de Verdi\`ere invariant or the paremeter $\sigma$ of van der Holst \& Pendavingh \cite{HolPenGra}. As a first  (or rather second) step in this direction one can ask 

\begin{question} \label{Q LE} 
Do the countable linklessly-embeddable graphs have a \mue?
\end{question}

For the definition of linklessly-embeddable graphs, and alternative characterizations of this family, we refer to \cite{RoSeThoSac}. I expect that these characterizations remain equivalent in the infinite case. % provide more than one ways to adapt the notion to infinite graphs, and it is not obvious that they coincide, but question \Qr{Q LE} is interesting for any of them.

Another approach for attacking \Cnr{conj cf} would be to try to employ the Graph Structure Theorem of Robertson \& Seymour \cite{GM17}. This raises the following question, for which our results of \Sr{sec K5} provide support:
\begin{problem} \label{ksums}
Are the families with \mue s closed under $k$-sums \fe\ $k$?
\end{problem} 
We remark that if $\cc$ has a \mue\ $U$, then $\mathrm{Apex(\cc)}$, i.e.\ the class of cones over graphs in \cc, has the cone over $U$ as a universal element. %\mymargin{conjecture that Apex has no \lf\ \mue?}

%SSSSSSSSSSSSSSSSSSSSSS
\comment{
	\subsection{Finite graphs} \label{sec fin}

The question of what is the smallest size of a graph containing every planar $n$-vertex graph as a subgraph has attracted considerable attention, see \cite{BCEGS,EsJoMoSpa} and references therein. I am not aware of any studies of the analogous question when replacing `subgraph' by `minor'. Given the fact that the planar graphs have a \mue\ but no subgraph-universal one, the following question arises naturally.

\begin{question}
What is the smallest size $S(n)$ (measured either in terms of vertices or edges) of a planar graph $G_n$ containing every planar graph of size $n$ as a minor?
\end{question}

Part of my motivation for this question lies in the extremal objects $G_n$: they might have an interesting geometry or converge to interesting limits.

	If we drop the requirement that $G_n$ be planar itself, then size $S(n)= \Theta(n)$ suffices. This is because every planar graph of size $n$ is a minor of a sub-cubic planar graph  of size $\Theta(n)$ by the construction in \Lr{lem lf}, and Capalbo \cite{CapSma} constructed a graph of size $O(n)$ containing each planar sub-cubic graph of size $n$ as a subgraph. It would still be interesting to compute a tight constant.
}
%SSSSSSSSSSSSSSSSSSSSSS
\subsection{\Cof\ graphs} \label{sec cof}

We say that a minor-closed graph class \cc\ is \defi{\finy}, if $\cc = \forb{S}$ for a set $S$ of finite graphs (which can be chosen to be finite by the \RSTc). Note that a graph \g belongs to such a class \cc\ \iff\ every finite minor of \g does. %The class of countable planar graphs is a good example.
Call an infinite graph \g \defi{\cof}, if $\cc_G:= \{H \mid H<G\}$ is \finy.

%every element of $\ex{G}$ is finite. 
It follows from the definitions that 
%%%%%%%%%
\begin{observation}  \label{obs cof}
A countable graph $G$ is \cof\  \iff\ $G$ is a \mue\ of a \finy\ minor-closed class \cc\ (with %$\ex{\cc_G}=\ex{\cc}$, and
 $\cc=\cc_G$).\qed
\end{observation}

By the Graph Minor Theorem \cite{GMXX}, there are at most countably many countable \cof\ graphs. Let us show that there are infinitely many. The regular tree $G_0$ of countably infinite degree is a \mue\ of the class of forests $\cc_0=\forb{K_3}$. It follows that \fe\ \nin, the graph $G_n$ obtained from $G_0$ by adding $n$ new vertices and connecting them completely will all other vertices is universal for the $n$-apex forests. Another infinite family can be obtained similarly, starting with $G_0$ being a universal planar graph.

These examples combined with the discussion of \Sr{sec Delta} motivate
\begin{question}  \label{Q cof}
Is there an infinite set of \lf\ \cof\ graphs no two of which are minor-twins?
\end{question}

A famous problem of Thomas \cite{ThoWel} asks whether the countable graphs are well-quasi-ordered under $<$. The restriction to \cof\ graphs may be more accessible:

\begin{problem}  \label{Q cof}
Are the countable \cof\ graphs well-quasi-ordered under $<$?
\end{problem}

%SSSSSSSSSSSSSSSSSSSSSS
\subsection{Addable classes} \label{sec add}

Recall from the introduction that  the class $\ce_\Sigma$ of countable graphs that embed in a closed surface $\Sigma \not\isom \BS^2$ has no \mue, but the class $\ce_\Sigma^*$ of graphs with components in  $\ce_\Sigma$ does. This section attempts to use these facts as a lesson to come up with a more concrete version of \Qr{Q DK}.

\medskip
A graph class \cc\ is called \defi{addable}, if $G,H\in \cc$ implies $G \dot{\cup} H \in \cc$. We say that \cc\ is \defi{\sig-addable}, if $G_n \in \cc$ for \nin\ implies $\dot{\bigcup}_{\nin} G_n \in \cc$.

\begin{proposition} \label{pr add}
A \finy\ minor-closed graph class \cc\ is addable \iff\ each $F\in \ex{\cc}$ is connected. 
An arbitrary minor-closed graph class \cc'\ satisfying $\cc=\forb{\ex{\cc'}}$\footnote{See \Sr{sec min} to recall how this could fail.} is \sig-addable \iff\ each $F\in \ex{\cc'}$ (is connected or) has a connected minor-twin. \qed
\end{proposition}
Variants of this for finite graphs are well-known.
The proof is an exercise that we leave to the reader. To see that the restriction to \finy\ classes is essential in the first statement, consider the class \cc\ of graphs embeddable in any closed surface. This \cc\ is addable (but not \sig-addable), but $\ex{\cc}$ contains %the disjoint union 
$\omega \cdot K_5$, which has no connected minor-twin. As another example, let $\cu$ be the class of graphs with at most one cycle. It is easy to see that \cu\ has no \mue, and that it is not addable. But $\cu^*$ has one: a disjoint union of cycles of all possible lengths, with copies of the \oo-regular  tree attached to each of their vertices.

\medskip
Motivated by the above discussion, and \Cr{cor}, we propose
\begin{problem} \label{fin add}
Let $F_1,\ldots, F_k$ be finite connected graphs. Must  $\forb{F_1,\ldots, F_k}$ have a \mue?
\end{problem} 

In other words, the question is whether every \finy, addable, minor-closed class has a \mue. 

\mymargin{how about $K_5 \vee K_5$ etc.?}

We remark that the \finy\ classes of countable graphs are in one-to-one correspondence with classes of finite graphs. Our final problem emphasizes the idea that some questions about infinite universal graphs can be thought of as questions about finite graphs: %In particular, it follows from (and Higman's Lemma \cite[Lemma~12.1.3.]{Diestelbook05}) that  by the subset relation. %setwise minor relation $<$; the latter is defined by letting $\cc < \cd$ if \fe\ $G\in \cc$ \ti\ $H\in \cd$ with $G<H$.

\begin{problem} \label{decide}
Let $F_1,\ldots, F_k$ be finite graphs. Is the problem of deciding whether $\forb{F_1,\ldots, F_k}$ has a \mue\  ---equivalently, the problem of whether there is a countable graph \g with $\ex{G}= \{F_1,\ldots, F_k \}$---  algorithmically decidable?
\end{problem}

%\newpage
\comment{
	\begin{lemma} \label{lem 3con}
Every countable planar graph is a minor of a \lf, 3-connected, planar graph.
\end{lemma}
%%%%%%%%%
\begin{proof}
Let \g be a countable planar graph. If \g is disconnected we can add edges between its components to make it connected, so let us assume that it is.  By \Lr{lem lf}, we may also assume \g is \lf. We claim that the graph $G^\otimes$ as in \Dr{def otimes} is 3-connected. Indeed, given $x,y\in  V(G)$, and a \pth{x}{y}\ $P$ in $\G$, notice that $P$ is blown up into a triple of independent \pth{x}{y}s in $G^\otimes$. For other couples of vertices $x,y\in  V(G^\otimes)$ a triple of independent \pth{x}{y}s can be found similarly.
\end{proof}
}

%\acknowledgement{ }

\bibliographystyle{plain}
\bibliography{collective}

\end{document}